\documentclass[11pt]{amsart}

\usepackage{amssymb, graphicx}

\usepackage{amsfonts}
\usepackage{latexsym}
\usepackage{amscd}

\newcommand{\tnz}{\otimes}
\newcommand{\mf}{\mathfrak}
\newcommand{\g}{\mf{g}}
\newcommand{\h}{\mf{h}}
\newcommand{\nl}{\mf{l}}
\newcommand{\so}{\mf{so}}
\newcommand{\gl}{\mf{gl}}
\newcommand{\sll}{\mf{sl}}
\newcommand{\p}{\mf{p}}
\newcommand{\q}{\mf{q}}
\newcommand{\tr}{{\rm tr}}
\allowdisplaybreaks[1]

\newtheorem{theorem}{Theorem}[section]
\newtheorem{proposition}[theorem]{Proposition}
\newtheorem{lemma}[theorem]{Lemma}
\newtheorem{corollary}[theorem]{Corollary}
\theoremstyle{remark}

\newcommand{\EE}{\mathcal E}
\newcommand{\LL}{\mathcal U}
\begin{document}
\title[On Lie and associative algebras containing inner derivations]{On Lie and associative algebras containing inner derivations}

\author{Matej Bre\v sar}
\author{\v Spela \v Spenko}
\address{M. Bre\v sar,  Faculty of Mathematics and Physics,  University of Ljubljana,
 and Faculty of Natural Sciences and Mathematics, University
of Maribor, Slovenia} \email{matej.bresar@fmf.uni-lj.si}
\address{\v S. \v Spenko,  Institute of  Mathematics, Physics, and Mechanics,  Ljubljana, Slovenia} \email{spela.spenko@imfm.si}

\begin{abstract} 
We describe  subalgebras of the Lie algebra  $\mf{gl}(n^2)$ that contain all inner derivations of  $A=M_n(F)$ (where $n\ge 5$ and $F$ is an algebraically closed field of characteristic $0$). In a more general context  where $A$ is a prime algebra satisfying certain technical restrictions, we establish a density theorem for the associative algebra generated by all inner derivations of $A$. 
\end{abstract}

\thanks{
2010 {\em Math. Subj. Class.} 15A30, 15A86, 16D60, 16N60, 16R60, 17B05, 17B10, 17B40.}
\thanks{Supported by ARRS Grant P1-0288.}

\maketitle
\section{Introduction}

In \cite{PD} Platonov and \DJ okovi\' c described algebraic subgroups of $GL(n^2)$, the group of invertible operators of the space $M_n(F)$ of all $n\times n$ matrices over an algebraically closed field $F$ with char$(F)=0$, that contain the group of all inner automorphisms of the algebra $M_n(F)$. After that they showed that this description can be effectively applied to  various linear preserver problems.

In the recent literature one can find  numerous results characterizing derivations through conditions analogous to those  satisfied by automorphisms and previously studied in the context of linear preserver problems.  Therefore it seems natural to ask whether
it is possible to describe subalgebras of $\mf{gl}(n^2)$, the Lie algebra of all operators of $M_n(F)$, that contain the Lie algebra of all
inner derivations, and thereby enable one to use a more conceptual approach to such problems on derivations. A partial description was  in fact obtained already by  Platonov and \DJ okovi\' c, but they  were more interested in corresponding Lie groups. In Section 2 we will complete their work and give a complete list of subalgebras of  $\mf{gl}(n^2)$ that contain inner derivations. Several corollaries will be derived in order to justify the usefulness of this result.

In Section 3 we will consider only the associative algebra $\mathcal D$ generated by all inner derivations, but on considerably more general algebras  than $M_n(F)$. A thorough discussion on $\mathcal D$, in both algebraic and analytic context,  was given in the  recent work by T. Shulman and
V. Shulman \cite{SS}. They were primarily interested in the form of operators in $\mathcal D$. Our approach is somewhat different. We will show that under suitable assumptions $\mathcal D$ acts densely on a  Lie ideal of the algebra in question.


\section{Lie algebras containing derivations}

\subsection{Notation}We begin by introducing the notation for this section. 
We write  $M_n$ for $M_n(F)$. Let us fix two assumptions that will be used  throughout the section:
 \begin{itemize}
 \item $F$ is   algebraically closed and  char$(F)=0$,
 \item $n\ge 5$.
 \end{itemize}
 Let us emphasize that these two assumptions will not be repeated in the statements of our results. 
 They both are connected with the Platonov-\DJ okovi\' c paper \cite{PD}. Actually, 
 \cite{PD} also deals with the situation where $n < 5$; however, this requires some extra care. For simplicity we will avoid this.

By $1$ we denote  the   identity matrix in $M_n$, by $e_{ij}$ the standard matrix units in $M_n$, and by 
 $a'$ the transpose of $a\in M_n$.  The Lie algebra $\mf{gl}(n^2)$ can be identified with $M_n\otimes M_n^{opp}$, where  the Lie bracket is given by
$$[a\otimes b,c\otimes d]=(a\otimes b)(c\otimes d)-(c\otimes d)(a\otimes b)=ac\otimes db- ca\otimes bd.$$
In this sense
$$\mf{g}=\{a\otimes 1-1\otimes a\,|\, a\in M_n \}$$
is equal to the Lie algebra of all inner derivations of $M_n$ (of course, all derivations on $M_n$ are inner).  We are interested in Lie subalgebras of $\mf{gl}(n^2)$ that contain $\mf{g}$.

We have a direct decomposition $M_n=M_n^0+F1$, where 
 $$M_n^0 = \{a\in M_n\,|\, \tr(a) =0\}.$$
 Clearly, every element in $\g$ can be uniquely written as $a\otimes 1-1\otimes a$ with $a\in M_n^{0}$, and 
 $\mf{g}$ is a simple Lie algebra isomorphic to $\mf{sl}(n)$. 
We denote by $\mf{gl}(n^2-1)$ the Lie subalgebra of $\mf{gl}(n^2)$ consisting of elements that preserve the decomposition and send $1$ into $0$. Next we set 
$$\mf{sl}(n^2-1)=\mf{sl}(n^2)\cap \mf{gl}(n^2-1),$$
$$\mf{so}(n^2) = \{a\in \mf{gl}(n^2)\,|\, \tr\bigl(a(x)y + xa(y)\bigr)=0\,\,\, \mbox{for all $x,y\in M_n$}\},$$ 
$$\mf{so}(n^2-1)=\mf{so}(n^2)\cap \mf{gl}(n^2-1).$$

Note that any Lie subalgebra of $\mf{gl}(n^2)$ containing $\mf{g}$ can be considered as a $\mf{g}$-module. We now state a list of simple $\mf{g}$-submodules of $\mf{sl}(n^2)$, as given in \cite[p. 170]{PD}.  By $\epsilon_i$ we denote the linear functional on diagonal matrices determined by $\epsilon_i(e_{jj})=\delta_{ij}$.
\begin{center}
\begin{tabular}{|c|c|c|c|}
\hline
Module&~~~Highest weight~~~&~~~Dimension~~~&~~~~~~~~Highest weight vector~~~~~~~~\\
\hline
$\mf{g}$& $\epsilon_1 - \epsilon_n$ & $n^2-1$ & $e_{1n}\otimes 1 - 1\otimes e_{1n}$\\ 
\hline
$V_1$& $\epsilon_1 + \epsilon_2 - 2\epsilon_n$ & $\frac{1}{4}(n^2-1)(n^2 -4)$ & $e_{1n}\otimes e_{2n}  - e_{2n}\otimes e_{1n}$\\ 
\hline
$V_2$& $2\epsilon_1 - \epsilon_{n-1} - \epsilon_n$ & $\frac{1}{4}(n^2-1)(n^2 -4)$ & $e_{1,n-1}\otimes e_{1n}  - e_{1n}\otimes e_{1,n-1}$\\ 
\hline
$V_3$& $\epsilon_1 - \epsilon_{n}$ & $n^2-1$ & $\sum_i ( e_{1i}\otimes e_{in}  - e_{in}\otimes e_{1i})$\\ 
\hline
$V_4$& $\epsilon_1 - \epsilon_n$ & $n^2-1$ & $e_{1n}\otimes 1  + 1\otimes e_{1n}$\\ 
\hline
$V_5$& $2\epsilon_1 - 2\epsilon_n$ & $\frac{1}{4}n^2(n-1)(n+3)$ & $e_{1n}\otimes e_{1n} $\\ 
\hline
$V_6$& $\epsilon_1 + \epsilon_{2} - \epsilon_{n-1}-\epsilon_n$ & $\frac{1}{4}n^2(n+1)(n-3)$ & $e_{2,n-1}\otimes e_{1n}  + e_{1n}\otimes e_{2,n-1}$ \\
&  & &$- e_{1,n-1}\otimes e_{2n} - e_{2n}\otimes e_{1,n-1} $\\ 
\hline
$V_7$& $\epsilon_1 - \epsilon_{n}$ & $n^2-1$ & $\sum_i (e_{1i}\otimes e_{in}  + e_{in}\otimes e_{1i})$\\ 
\hline
$V_8$& $0$ & $1$ & $1\otimes 1 - n\sum_{i,j}  e_{ij}\otimes e_{ji} $\\ 
\hline
$\mf{p}$& $\epsilon_1 - \epsilon_{n}$ & $n^2-1$ & $\sum_i e_{1i}\otimes e_{in} $\\ 
\hline
$\mf{q}$& $\epsilon_1 - \epsilon_{n}$ & $n^2-1$ & $\sum_i e_{in}\otimes e_{1i} $\\ 
\hline
$V_4'$& $\epsilon_1 - \epsilon_{n}$ & $n^2-1$ & $n(e_{1n}\otimes 1 + 1\otimes e_{1n})$ \\
&  & &$ - 2\sum_i (e_{1i}\otimes e_{in} + e_{in}\otimes e_{1i})$\\ 
\hline
\end{tabular}
\end{center}


By $T_0$ we denote the set of all diagonal matrices of the form $$\left( \begin{array}{cc}
\alpha 1_{n^2 -1}& 0 \\
 0 & \beta \end{array} \right),$$
where $\alpha,\beta \in F$ and $1_{n^2 -1}$ is the identity matrix in $M_{n^2-1}$. Its subset consisting of all such matrices with $\alpha,\beta\in F^*$ will be denoted by $T$. 
Clearly, $T$ is a group.

The notation introduced so far is taken from \cite{PD}. 
Let us introduce
another $\g$-module that will appear in the course of the proof of Proposition \ref{list} below. 
As we shall see, it is possible to describe it in terms of $V_4$, $\p$, and $\q$, but first we  give a more explicit description. 
Take $\lambda \in F$, $\lambda\ne -\frac{2}{n}$, and set
$$
W(\lambda) = \{x\mapsto ax + xa + \lambda \tr(x)a - \frac{2\lambda}{n\lambda + 2}\tr(xa)1\,|\,a\in M_n^0\}.
$$
One can verify that $\g + W(\lambda)$ is a Lie algebra, and implicitly we will in fact make this verification in the proof of  Proposition \ref{list}.

Let us point out that we will use both symbols $\subseteq $ and $\subset$. The latter will be of course used to denote a proper subset.

\subsection{Extracts from the Platonov-\DJ okovi\' c paper} This section rests heavily on the work by  Platonov and \DJ okovi\' c \cite{PD}. 
We will now record several facts that are more or less explicitly stated in the proof of \cite[Theorem A]{PD}. They will be used in the next subsections.

In addition to the notation introduced above, we let 
 $\mf{l}$ denote a Lie subalgebra of $\mf{sl}(n^2)$ that contains $\mf{g}$. Therefore $\mf{l}$ can be treated as a $\mf{g}$-module.

\begin{enumerate}
\item\label{dec}
 $\sll(n^2)$, considered as a $\mf{g}$-module, can be directly decomposed into simple
 $\g$-modules as follows:
$$\sll(n^2)=\g+\sum_{i=1}^8 V_i.$$

\item \label{iso} The   $\g$-modules $\g,V_3,V_4,V_7,\p,\q,V_4'$ are isomorphic, while the others from the above list  are pairwise nonisomorphic
(cf. \cite[Theorem 20.3A]{H}).


\item\label{dec2}
We have the following direct decompositions  into simple  $\g$-modules:
$$\so(n^2-1)=\g+V_1+V_2,$$
$$\so(n^2)=\so(n^2-1)+V_3,$$
$$\sll(n^2-1)=\so(n^2-1)+V_4^{'}+V_5+V_6,$$
$$\sll(n^2)=\sll(n^2-1)+\p+\q+V_8.$$

\item\label{ort}
 $\so(n^2)$ is the Lie subalgebra of $\mf{gl}(n^2)$ consisting of all skew-symmetric tensors.

\item\label{max}
If  $\mf{l}$ is properly contained in 
 $\sll(n^2-1)$ and properly contains $\g$, then  $\mf{l}=\so(n^2-1)$.

\item\label{v5,v6}
If $V_5\subset \nl$ or $V_6\subset \nl$, then $\sll(n^2-1)\subseteq \nl$.



\item\label{p+q}  If $\p+\q\subset \nl$, then $\nl=\sll(n^2)$.

\item\label{v8}
 $V_4^{'},\p,\q$ are pairwise nonisomorphic as $(\g+V_8)$-modules.

\item\label{v8'}
If $V_8\subset \nl$ and $\p+\q\not\subset \nl$, then $\nl\subseteq \sll(n^2-1)+\p+V_8$ or $\nl\subseteq \sll(n^2-1)+\q+V_8$.

\item\label{v1,v2}
If $V_1\subset \nl$ or $V_2\subset \nl$, then $\so(n^2-1)\subseteq \nl$.

\item\label{iso2}
As an $\so(n^2-1)$-module, $\sll(n^2)$ is a direct sum of five simple modules, namely $$\so(n^2-1),\,V_4^{'}+V_5+V_6,\,\p,\,\q,\,V_8.$$
 Among them only $\p$ and $\q$ are isomorphic.

\item\label{w}
If $W$ is a simple submodule of $\p+\q$, different from $\p$ and $\q$, 
then there exists $t\in T$ such that $W=tV_3t^{-1}$.

\item\label{3} If $W$ is a $\g$-submodule of $\p+\q$ and $W\neq \p,\q$, then $\g+W$  is not a Lie algebra. Next, if
$\nl\subset \g+V_4+\p+\q$, then $\nl$ can not be written as a sum of three simple $\g$-modules.

\end{enumerate}

The $\g$-modules $\p,\q, V_4$ and $V_8$ will  play particularly prominent roles. Therefore we will now  give  some further comments about them, which will be used in the proof of  Proposition \ref{list} without reference.

With respect to the  decomposition $M_n=M_n^0+ K1$ we can represent 
 $\p$ and $\q$  with matrices of the form 
$$
\left( \begin{array}{cc}
0 & x \\
0 & 0 \end{array} \right)\mathrm{and}
\left( \begin{array}{cc}
0 & 0 \\
x' & 0 \end{array} \right)\mathrm{, respectively},
$$
where $x\in F^{n^2-1}$ is an arbitrary column vector.
Thus, $\p$ consists of all maps of the form 
$x\mapsto \tr(x)a$ with $a\in M_n^0$, 
and $\q$ consists of all maps of the form 
$x\mapsto \tr(xa)1$ with $a\in M_n^0$.
Next, $V_8$ consists of scalar multiples of the diagonal matrix $$\left( \begin{array}{cc}
1_{n^2 -1}& 0 \\
 0 & 1-n^2 \end{array} \right),$$
 and $V_4=\{a\otimes 1 + 1  \otimes a \,|\,a\in M_n^0\}$. 

As already mentioned, $\g,\p,\q$, and $V_4$ are isomorphic simple $\g$-modules. Since $F$ is algebraically closed,  
up to scalar multiplication there is exactly one isomorphism from $\g$ into any of modules $\p,\q$, and $V_4$. One can check that these are
$$\phi_1:\g\to \p,\; \phi_1(a\tnz 1-1\tnz a)(x)=\tr(x)a,$$ 
$$\phi_2:\g\to \q,\; \phi_2(a\tnz 1-1\tnz a)(x)=\tr(xa),$$ 
$$\phi_3:\g\to V_4,\;  \phi_3(a\tnz 1-1\tnz a)=a\otimes 1+1\otimes a.$$

Finally, since the highest weight vector of  $V_4'$ is a linear combination of the highest weight vectors of  $\p$, $\q$ and $V_4$, it easily follows that $\p+\q+V_4=\p+\q+V_4'$. 


\subsection{Main result} Besides the results from the previous subsection we also  need the following well-known fact (see e.g. \cite[Exercise 17, p. 124]{hall}).
Let $\mathfrak{h}$ be a  Lie algebra and let $V$ be a finite dimensional $\mathfrak{h}$-module. Assume that $V$ is a direct sum of simple $\mathfrak{h}$-modules, 
$$V=V_1^1+\cdots+ V_{n_1}^1+ V_1^2+ \cdots + V_{n_2}^2+ \cdots + V_1^k+ \cdots + V_{n_k}^k,$$
where $V_i^j$ is isomorphic to $ V_{i'}^j$ for all $i,i',j$ and  nonisomorphic to $V_l^k$ for all $l$ and $k\neq j$. 
Then every simple submodule $U$ of $V$ is contained in  $V_1^j+\cdots+ V_{n_j}^j$ for some $j$, 
and there exist $\lambda_1,\ldots,\lambda_ {n_j}\in F$, not all zero, such that $U$ consists of all elements of the form 
$$\lambda_1v+\lambda_2 \phi_2(v)+ \ldots + \lambda_ {n_j}\phi_{n_j}(v),\,\, v\in V_1^j,$$ where   $\phi_l:V_1^j\to V_l^j$ is an $\mathfrak{h}$-module isomorphism.
We will use this fact frequently and without reference.

Our goal is to give a list of all Lie subalgebras of $\gl(n^2)$ that contain $\g$. In the first and major step we describe those of them that are contained in $\sll(n^2)$.

\begin{proposition}\label{list}
If $\nl$ is a proper Lie subalgebra of $\sll(n^2)$ that contains $\g$, then  $\nl$ is one of the following Lie algebras:
\begin{enumerate}
\item[(i)]
$\sll(n^2-1),\sll(n^2-1)+\p,\sll(n^2-1)+\q,\sll(n^2-1)+V_8, \sll(n^2-1)+\p+V_8,\sll(n^2-1)+\q+V_8,$ 
\item[(ii)]
$\so(n^2-1),\so(n^2-1)+\p,\so(n^2-1)+\q,\so(n^2-1)+V_8, \so(n^2-1)+\p+V_8,\so(n^2-1)+\q+V_8,$ 
\item[(iii)]
$\g,\g+\p,\g+\q,\g+V_8, \g+\p+V_8,\g+\q+V_8,$
\item[(iv)] 
$t\so(n^2)t^{-1}$ for some $t\in T$,
\item[(v)] 
$\g+W(\lambda)$  for some $\lambda\in F$, $\lambda\ne -\frac{2}{n}$.
\end{enumerate}

\end{proposition}

\begin{proof}
Consider $\nl$  as a $\g$-module. Since $\g$ is simple, $\nl$ is a  direct sum of simple modules by Weyl's theorem. 
As $\nl$ is a $\g$-submodule of  $\sll(n^2)=\g+\sum_{i=1}^8 V_i$ (see \eqref{dec}), 
it follows from  (\ref{iso}) that $\nl$ is equal to a sum of $\g$, some of the modules $V_1,V_2,V_5,V_6,V_8$, and 
a submodule of $V_3+V_4+V_7$. 
We have to figure out when such a sum forms a Lie algebra  and can be therefore equal to $\nl$.
The proof is divided into several steps.
First we consider the case where $\nl$  contains $\sll(n^2-1)$, which leads to the list (i).  
The case where $\sll(n^2-1)\not\subseteq\nl$ is more complex. After finding the Lie algebras  from  (ii) and (iii) we consider separately the situation  where $\nl$ contains   
 $\so(n^2-1)$
and the situation where it does not. This yields (iv) and (v), respectively.

So first suppose that $\nl$ contains $\sll(n^2-1)$. From (\ref{dec2}) we see that $\sll(n^2)=\sll(n^2-1)+\p+\q+V_8$. Therefore
$\nl=\sll(n^2-1)+Z$, where $Z$ is a $\g$-module contained in $\p+\q+V_8$. Since $\p$ and $\q$ are isomorphic \eqref{iso}, the
 only possibilities for $Z$ are $$0,\p,\q,V_8,\p+V_8,\q+V_8,\p+\q,W,W+V_8,$$
  where $W$ is a proper submodule of $\p+\q$ different from $0,\p,\q$.
We have to find out for which of these nine choices $\sll(n^2-1)+Z$ is a Lie algebra. 
 It is easy to check that this is true for the first six ones.  On the other hand, the last three choices must be ruled out in view 
 of  (\ref{p+q}), (\ref{w}), and (\ref{v8'}), respectively. Thus, (i) lists all  proper Lie subalgebras of $\sll(n^2)$  that contain  $\sll(n^2-1)$.


From now on we assume that $\nl$ does not
 contain $\sll(n^2-1)$.
Note that $V_5$ and $V_6$ are not contained in $\nl$ due to (\ref{v5,v6}). As $$\sll(n^2)=\so(n^2-1)+V_4'+V_5+V_6+\p+\q+V_8$$
 by (\ref{dec2}), it follows that 
$\nl\subset \so(n^2-1)+V_4'+\p+\q+V_8$  (the strict inclusion holds because of  (\ref{p+q})).
Moreover, (\ref{v1,v2}) and $\so(n^2-1)=\g+V_1+V_2$ (see (\ref{dec2})) imply that 
$$\nl=\g+W\quad\mbox{ or}\quad \nl=\so(n^2-1)+W,$$ where $W$ is a submodule of $V_4'+\p+\q+V_8$. 
It is easy to check that  $0$, $\p$, $\q$, $V_8$, $\p+V_8$, $\q+V_8$ are appropriate choices for  $W$; 
that is, $\so(n^2-1)+W$ and $\g+W$ are indeed Lie algebras in each of these cases. 
From now on we assume that $W$ is different from these modules. In other words, we are assuming that $\nl$ is none of the Lie algebras listed in (ii) and (iii).

Assume that $V_8\subset \nl$.  Then we have $W=V_8+W'$, where $W'$ is a submodule of $V_4'+p+q$.  According to (\ref{v8}) (and (\ref{p+q})), $W'$ can be equal to one of 
$\p,\q,V_4'$ or to a sum of two of them.
If $W'\in \{\p,\q\}$, then $\nl$ can be found in one  of the lists (ii) or (iii). On account of (\ref{p+q}), the only additional examples can be obtained if $V_4'\subseteq W'$. Thus,
 either  $\so(n^2-1) + V_4'\subset \nl$ or  $\g + V_4'\subset \nl$. However, from (\ref{max}) (and \eqref{dec2}) it follows that in each of these two cases $\nl$ contains  
$\sll(n^2-1)$,  contrary to our assumption.
Therefore  $V_8\not\subset \nl$.

Assume now that $\nl$ contains $\so(n^2-1)$, so that $\nl=\so(n^2-1)+W$. Since the Lie algebras from (ii) and (iii) have already been excluded, we see from 
 (\ref{iso2}) that  $W$ is a  submodule of $\p+\q$ different from $0,\p,\q$ and isomorphic to $\p\cong \q$. 
 By (\ref{w}), $W=tV_3 t^{-1}$ for some $t\in T$. Since
 $\so(n^2)=\so(n^2-1)+V_3$ (see \eqref{dec2}), it follows that $\nl = t\so(n^2)t^{-1}$.
 On the other hand, $t\so(n^2)t^{-1}$ indeed is a Lie algebra for every $t\in T$. We have thus arrived at the case (iv).
 

Finally  we consider the case where $\nl=\g+W$ with $W\subset V_4'+\p+\q$. As mentioned above, this is equivalent to  $W\subset V_4+\p+\q$.
In view of 
(\ref{3}) and (\ref{p+q}) we may assume that $W$ is a simple module and  $W\not\subset \p+\q$. Accordingly, there exist $\lambda,\mu\in F$ such that $W$ consists of all maps of the form
$$
w_a: x\mapsto ax + xa +\lambda\tr(x)a+\mu\tr(xa)1,\,\,\,\mbox{where $a\in M_n^0$.}
$$
The fact that $\nl$ is a Lie algebra gives rise to some restrictions on $\lambda$ and $\mu$. Indeed, we have
$$[w_a,w_b](x)=[[a,b],x] +(n\lambda\mu+2\lambda+2\mu)\bigl(\tr(xb)a-\tr(xa)b\bigr)\in \nl = \g + W$$ 
Since $x\mapsto [[a,b],x]$ lies in $\g\subset \nl$, it follows that either $n\lambda\mu+2\lambda+2\mu=0$ or $x\mapsto \tr(xb)a-\tr(xa)b$ lies in $\nl$ for all $a,b\in M_n^0$. The rank of this operator is at most $2$. On the other hand, operators in $\nl$ are of the form
$$
h:x\mapsto [c,x]+ dx+xd + \lambda\tr(x)d+\mu\tr(xd)1 = ex + xf + \lambda\tr(x)d+\mu\tr(xd)1,
$$ 
where $e = d+c$ and $f=d-c$. An elementary argument (see, e.g., \cite{Kuc}) shows that the  
 rank of $x\mapsto ex+xf$ is at least $n$, unless $e=-f$ is a scalar matrix (and hence $h=0$).
Accordingly, nonzero elements in $\nl$ have rank at least $n-2$. Since $n\ge 5$, $\nl$ cannot contain operators $x\mapsto \tr(xb)a-\tr(xa)b$. Therefore $n\lambda\mu+2\lambda+2\mu=0$, showing that
 $\lambda\ne -\frac{2}{n}$ and $W=W(\lambda)$. That is, $\nl$ is of the form described in (v).
\end{proof}

The main result of this section now follows easily.

\begin{theorem}\label{prop2}
If $\h$ is a proper Lie subalgebra of $\mf{gl}(n^2)$ that contains $\g$,  then either $\h=\sll(n^2)$ or $\h=\nl +F t$, where  $\nl$ is a Lie algebra from Proposition \ref{list} and $t\in T_0$. Moreover, if $\nl$ is a Lie algebra from {\rm (iv)} or {\rm (v)}, then $t=0$ or $t=1$.
\end{theorem}

\begin{proof}
In view of Proposition \ref{list} 
we may assume that $\h$ is not contained 
in $\sll(n^2)$.  Thus,  
$\nl=\h\cap \sll(n^2)$ is a proper Lie subalgebra of $\h$. Since $\sll(n^2)$ has codimension 1 in $\gl(n^2)$, $\nl$  has codimension 1 in $\h$.
Also, $\nl$ is a 
 $\g$-module, and so $\h=\nl + U$ for some 1-dimensional $\g$-submodule $U$ of $\mf{gl}(n^2)$. 
In the decomposition (\ref{dec}) there is only one $1$-dimensional module, namely $V_8$, so $U$ is  a submodule of $V_8+F1$. Therefore
$U = Ft$, where 
$t$ is a matrix of the form
$ \left( \begin{array}{cc} \alpha 1_{n^2-1}& 0 \\
 0 & \beta \end{array} \right)\in T_0$ for some $\alpha,\beta\in F$. It is easy to see that for any choice of $\alpha$ and $\beta$, $\h=\nl + Ft$ is a Lie 
 algebra if $\nl$ is listed in (i), (ii), or (iii). On the other hand, a brief examination shows that $\alpha$ must be equal to $\beta$ if $\nl$ is 
 listed in (iv) or (v).
  \end{proof}

\subsection{Applications}  Theorem \ref{prop2} makes it possible for us to find various conditions that are characteristic for derivations. Here is a sample result:

\begin{corollary}
If $d:M_n\to M_n$ is a linear map such that $d(1)=0$ and 
$$\tr\Bigl(d(x)yz+xd(y)z+xyd(z)\Bigr)=0 \quad\mbox{for all $x,y,z\in M_n$},$$
then $d$ is a derivation. 
\end{corollary}

\begin{proof}
It is easy to verify that the set $\h$ of all maps $d$ that satisfy the conditions of the corollary is a Lie algebra 
that contains $\g$. Therefore $\h$ can be found among Lie algebras listed in Theorem \ref{prop2}. If we put $z=1$ in the above identity
we get $\h\subseteq \so(n^2)$. Further, setting $y=z=1$ we see that  $\h$ preserves the decomposition $M_n= M_n^0 + F1$. The only possible choices for $\h$ 
are therefore $\g$ and 
$\so(n^2-1)$. It remains to find an element of $\so(n^2-1)$ that does not satisfy the condition of the corollary. An example is
$e_{11}\tnz e_{22}-e_{22}\tnz e_{11}$ together with $x=e_{12}, y=e_{23}, z=e_{31}$.
\end{proof}

Kernels of derivations are obviously associative subalgebras. This property is characteristic for derivations in the following sense:

 \begin{corollary}\label{cor2}
If $\nl$ is a Lie subalgebra of $\mf{gl}(n^2)$ that contains $\g$ and has the property that  $\ker (d)$ is an associative subalgebra of 
$M_n$ for every $d\in \nl$, then $\nl=\g$.
\end{corollary}

\begin{proof}
We have to check that  each of the Lie algebras listed in Theorem 
\ref{prop2},  except $\g$, contains an operator whose kernel is not a subalgebra.  
 It is enough to show that $\so(n^2-1)$, $\p$, $\q$, $W(\lambda)$ and $\g+Ft$ with $t\in T_0\setminus \{0\}$ contain such operators. Namely, 
all  Lie algebras from the list, except $\g$, contain at least one of these sets.

\begin{itemize}
\item
In $\so(n^2-1)$ consider $r=e_{12}\otimes e_{34}-e_{34}\otimes e_{12}$ (cf. \eqref{ort}). Then $e_{21},e_{13}\in \ker(r)$, while 
$e_{21}e_{13}=e_{23}\not\in \ker(r)$.
\item
The kernel of every nonzero element of $\p$ is equal to $M_n^0$, which is not a subalgebra of $M_n$ .
\item
There is a variety of examples in $\q$, say $\sum_{i=1}^n e_{i1}\tnz e_{2i}$.
\item
Observe that $x=e_{11}-e_{22}+\sqrt{-1}(e_{33}-e_{44})$ lies in the kernel of $w_{e_{12}}\in W(\lambda)$. However, this does not hold for $x^2$.
\item
Maps in $\g+Ft$ are of the form  $r_a:x\mapsto[x,a]+\alpha x+\beta \tr(x)1$, where $a\in M_n$, $\alpha,\beta\in F$, $\alpha\ne 0$ or $\beta\ne 0$.
If $\alpha\neq 0$, take $a=\alpha(e_{11}-e_{33})$. Then $e_{12},e_{23}$ are elements from $\ker(r_a)$, while this does not hold for their product 
$e_{13}$. If $\alpha=0$, choose $a=e_{12}$. Then
$e_{34},e_{43}$ lie in $\ker(r_a)$, but not their product $e_{33}$.
\end{itemize}
\end{proof}

Let $f=f(\xi_1,\ldots,\xi_l)$, $l\ge 2$, be a multilinear polynomial in noncommuting variables, i.e.,
$$
f = \sum_{\sigma\in S_l} \lambda_\sigma \xi_{\sigma(1)}\ldots \xi_{\sigma(l)} 
$$
where $\lambda_\sigma\in F$ and $S_l$ is a symmetric group. A linear map $d$ from an algebra $A$ into itself is said to be an 
 {\em $f$-derivation} if 
 \begin{equation}\label{f-odv} \tag{$\ast$}
d\bigl(f(x_1,\dots,x_l)\bigr)=\sum_{i=1}^{l}f(x_1,\dots,x_{i-1},d(x_i),x_{i+1},\dots,x_l) 
\end{equation}
for all $x_1,\ldots,x_l\in A$. Derivations are obvious examples, and the question is whether they are basically also the only possible examples.   An affirmative answer has been obtained for rather general algebras (see, e.g., \cite[Section 6.5]{FIbook}), but, surprisingly, the case where $A=M_n$ still has not been completely settled. Under certain technical restrictions we are now in a position to handle it.  To exclude some pathological cases we  assume that  $d(1)=0$. We also assume that $l < 2n$ since otherwise $f$ could be a polynomial identity of $M_n$, making \eqref{f-odv}  meaningless.
It should be mentioned that the arguments in the next proof are similar to those from the recent paper \cite{ABSV}, in which the Platonov-\DJ okovi\' c theory was applied.

\begin{corollary}
Let $f$ be a multilinear polynomial of degree $ l<2n$.
If $d:M_n\to M_n$ is an $f$-derivation such that $d(1)=0$, then  $d$ is a derivation.
\end{corollary}

\begin{proof}
Note that the set  $\h$  of all $f$-derivations of $M_n$ is a  Lie subalgebra of $\mf{gl}(n^2)$ that contains $\g$. 
Hence $\h$ is one of the Lie algebras listed in Theorem \ref{prop2}.  As in the proof of Corollary \ref{cor2} it suffices to show that 
 $\h$ does not contain  $\so(n^2-1)$, $\p$, $\q$, $W(\lambda)$ and $\g+Ft$ with $t\in T_0\setminus \{0\}$. 
The second and the fourth  possibility can be cancelled out due to the initial assumption  $d(1)=0$.  

We claim that $\so(n^2-1)\not\subseteq \h$. Without loss of generality we may assume that $x_1\ldots x_l$ is one of the monomials of $f$.
Choose $e_{13}\tnz e_{22}-e_{22}\tnz e_{13}\in \so(n^2-1)$ (cf. \eqref{ort}). If $l=2k-2$ (resp. $l=2k-1$), 
take $(x_1,\dots,x_n)=(e_{11},e_{32},e_{22},e_{23},\dots,e_{k-1,k})$ (resp. $(x_1,\dots,x_n)=(e_{11},e_{32},e_{22},e_{23},\dots,e_{k,k})$).
Then observe that in this case the left-hand side of (\ref{f-odv}) differs from its right-hand side. This proves our claim.

The task now is to exclude the case $\q\subset \h$. Note that $\sum_{i=1}^ne_{i1}\tnz e_{ki}\in \q$ with
$(x_1,\dots,x_n)=(e_{11},e_{12},e_{22},e_{23},\dots,e_{k-1,k})$ (resp. $(x_1,\dots,x_n)=(e_{11},e_{12},e_{22},e_{23},\dots,e_{k,k})$), 
depending on the parity of $l$, does the trick. 

We are reduced to proving that  $t\in T_0\setminus{\{0\}}$ cannot belong to $\h$.
We can restrict ourselves to the case where $t$ acts as a scalar multiple of the identity on $M_n^0$ and sends $1$ to $0$.
Now choose a maximal subset $S$ of $\mathbb{N}_l=\{1,\dots,l\}$ such that  the polynomial
$f(y_1, \ldots,y_l)$, where $y_i= 1$ if $i\in S$ and $y_i =x_i$ if $i\notin S$, is not zero (the case where $S=\emptyset$ is not excluded). Since the degree
of $f(y_1,\ldots,y_l)$ is less than $2n$, this polynomial is not an identity of $M_n$. Therefore there exist $a_1,\ldots,a_l\in M_n$ such that $a_i=1$ if $i\in S$
and $f(a_1,\ldots,a_l)\ne 0$. Moreover, because of the maximality assumption we may assume that $a_i\in M_n^0$ whenever $i\notin S$. Note that 
\eqref{f-odv} yields
$$f(a_1,\dots,a_l)-\frac{1}{n}\tr\bigl(f(a_1
,\dots,a_l)\bigr)1=(l-s)f(a_1,\dots,a_l),$$
where $s=|S|$.
This is possible only when $l-s=0$ or $l-s=1$. Actually, from the definition of $S$ it is clear that the last possibility cannot occur.
Therefore $l=s$. Considering $f(a,\ldots,a)$ for an arbitrary $a\in M_n^0 $ we easily derive a contradiction.
\end{proof}

Some of the Lie algebras from Theorem \ref{prop2} (and Proposition \ref{list}) are also closed under the associative product, 
and are therefore associative 
algebras. In the next corollary we will list all of them. Although the symbols such as $\gl$ etc. are traditionally reserved for Lie algebras, we will 
slightly abuse the notation and consider them as associative algebras. 

\begin{corollary}\label{gln}
If $A$ is a proper associative subalgebra of $\gl(n^2)$ that contains $\g$, then $A$ is either 
$$\gl(n^2-1),\,\gl(n^2-1)+\p,\,\gl(n^2-1)+\q,$$
$$\,\gl(n^2-1)+Ft, \,\,
\gl(n^2-1)+\p+Ft,\,\,\,\mbox{or}\,\,\,\gl(n^2-1)+\q+Ft$$
for some  $t\in T$.
\end{corollary}

\begin{proof}
All we have to do is to find out which of the Lie algebras from Theorem \ref{prop2} are closed under the associative product.
Take elements $e_{12}\tnz 1-1\tnz e_{12},e_{34}\tnz 1-1\tnz e_{34}\in \g$. 
Their product $u=-e_{12}\tnz e_{34} -e_{34}\tnz e_{12}$ preserves the decompostion 
$M_n=M_n^0+K1$ and has zero trace, thus it lies in $A\cap \gl(n^2-1)\cap \sll(n^2)$. 
Note that $\nl\cap \gl(n^2-1)\cap \sll(n^2)=\nl\cap \sll(n^2-1)$ for $\nl$ listed in Proposition \ref{list} (i), (ii), (iii), (iv), (v) is equal to $\sll(n^2-1), \so(n^2-1),
\g,\so(n^2-1),\g$, respectively. 
But $u$  lies neither in  $\so(n^2-1)$ nor in 
$\g$. Hence $\sll(n^2-1)\subset A$. Therefore $A$ also contains $\gl(n^2-1)$, which is 
 the associative algebra generated by $\sll(n^2-1)$.
 All Lie algebras from Theorem \ref{prop2} that contain $\gl(n^2-1)$ are indeed associative algebras. These are the algebras listed in
 the statement of 
 the corollary.
\end{proof}

 \section{Associative algebras generated by derivations}

 \subsection{Preliminaries}
 Throughout this section we assume that 
 \begin{itemize}
 \item $F$ is a field with char$(F)\ne 2$.
 \end{itemize}
 Let $A$ be an  $F$-algebra. By $Z(A)$ we denote its center. 
 Recall that $A$  is said to be {\em prime} if the product of any  two nonzero ideals of $A$ is  nonzero.
 
 We begin with a simple lemma which can be found at different places in the literature. Anyway, we give a proof as it is very short.
 
 \begin{lemma}\label{LL}
 If $A$ is a prime algebra and $m\in A$ is such that $[m,A]\subseteq Z(A)$, then $m\in Z(A)$. 
 \end{lemma} 
 
 \begin{proof}
Our assumption implies that $ [m,x]m = [m,xm]  \in Z(A)$  
for every $x\in A$. Since $[m,x]$ also lies in $Z(A)$, it follows that 
 $[m,x]^2 = [[m,x]m,x] =0$. As the center of a prime algebra cannot contain nonzero nilpotent elements, this gives $m\in Z(A)$. 
   \end{proof}

 We will consider prime $F$-algebras $A$  that contain a unique minimal ideal, i.e., a nonzero ideal $M$ that is contained in every nonzero ideal of $A$. 
  This class of algebras includes two important subclasses: simple algebras (in this case $M=A$) and primitive algebras with minimal one-sided ideals (in this case $M$ is the socle of $A$). 
 
 Recall that a subspace  $L$ of $A$ is said to be a {\em Lie ideal} of $A$ if $[L,A]\subseteq L$. Every subspace of  $Z(A)$  is obviously a Lie ideal. However, we are interested in {\em noncentral} Lie ideals, i.e., such that are not subsets of $Z(A)$.   
The next result is basically due to Herstein.

\begin{theorem}\label{HerLie2}
If $M$ is the unique minimal ideal of a noncommutative prime algebra $A$, then $[M,A]$ is the unique minimal noncentral Lie ideal of $A$. 
\end{theorem}

 \begin{proof}
 It is clear that $[M,A]$ is a Lie ideal of $A$. Suppose that $[M,A]\subseteq Z(A)$. Lemma \ref{LL} 
shows that in this case 
$M\subseteq Z(A)$. Hence $m,mx\in Z(A)$ for all $m\in M$, $x\in A$, yielding
 $m[x,y] = [mx,y] =0$ for every $y\in A$. Since $A$ is prime and noncommutative, this readily implies $M=0$, contrary to our assumption. Therefore $[M,A]$ is not central. 
 
 The fact that $[M,A]$ is the unique minimal noncentral Lie ideal follows from the well-known theorem saying that every noncentral Lie ideal of a prime algebra (over a field of characteristic not $2$) contains a Lie ideal of the form $[J,A]$, where $J$ is a nonzero ideal of $A$. 
 A possible reference is  \cite[Theorem 2.5]{BKS}; however,  the theorem  should be attributed to Herstein \cite{Her} (although it is not  explicitly stated therein).  
  \end{proof}

We will additionally assume that our algebra is {\em centrally closed}, meaning that its extended centroid is equal to $F$. We refer to the book \cite{BMMb} for a full account of this notion. Let us just say here that every simple unital ring $A$ can be viewed as a centrally closed algebra over $Z(A)$, and that primitive algebras with minimal one-sided ideals are centrally closed under  natural assumptions (cf. \cite[Theorem 4.3.7]{BMMb}). 

The next result is a special case of the main theorem of \cite{BM}. We remark that \cite{BM}  was one of the early papers on functional identities. Using the advanced functional identities  theory, as surveyed in \cite{FIbook}, the proof  
could now be somewhat shortened.

\begin{theorem}\label{bob} 
Let $A$ be a centrally closed prime $F$-algebra, let $L$ be a noncentral Lie ideal of $A$, and let $f:L\to L$ be an additive map such that $[f(u),u] =0$ for all $u\in L$.
Then there exist $\lambda\in F$ and a map $\mu:L\to Z(A)\cap L$ such that $f(u)=\lambda u + \mu(u)$ for every $u\in L$.
\end{theorem} 
 


Besides these three results, we will make use of the Jacobson Density Theorem. 
 
 \subsection{Density theorem for the algebra generated by inner derivations}
 Corollary \ref{gln} shows that the algebra generated by inner derivations of $M_n$ can be identified with the algebra of all operators on the space $M_n^0=[M_n,M_n]$. We will now generalize this result. Recall that an algebra $\mathcal A$ of linear operators on a space $X$ is said to act densely on $X$  if for all linearly independent $x_1,\ldots,x_n\in X$ and all $y_1,\ldots,y_n\in X$ there exists  $T\in\mathcal A$ such that $Tx_i=y_i$, $i=1,\ldots,n$.
 
 \begin{theorem}\label{T2}
 Let $A$ be a centrally closed prime $F$-algebra. Suppose that $A$ has a  unique minimal ideal $M$. If  $Z(A)\cap [M,A] =0$,
 then the subalgebra $\mathcal D$  of {\rm End}$_F(A)$ generated by all inner derivations of
  $A$ acts densely on 
 $[M,A]$. 
 \end{theorem}

\begin{proof}
Pick a nonzero $m\in [M,A]$. By assumption, $m\notin Z(A)$.  Note that $\mathcal D m $ is a Lie ideal of $A$.
Lemma \ref{LL} shows that it is noncentral. But then  $\mathcal D m =[M,A] $  by
  Theorem \ref{HerLie2}. This shows that  $[M,A]$ is a simple $\mathcal D$-module. Take $\delta\in {\rm End}_\mathcal D([M,A])$. Then, in particular, $ ({\rm ad}\, u)( \delta(u)) =\delta(({\rm ad}\, u)(u)) = 0$ for all $u \in [M,A]$. 
  That is, $[\delta(u),u]=0$. Theorem \ref{bob} shows that there is $\lambda\in F$ such that $ \delta(u)=\lambda u$. Thus ${\rm End}_\mathcal D([M,A])\cong F$ and the desired conclusion follows immediately from the Jacobson Density Theorem.
  \end{proof}
 
We remark that the algebra $K(X)$ of all compact operators on a Banach space $X$, just as any subalgebra of $B(X)$ that contains all finite rank operators, is covered by Theorem \ref{T2}. We have mentioned this because \cite{SS} devotes a special attention to $\mathcal D$ in the case where $A=K(X)$.

Recall that an $F$-algebra $A$ is said to be {\em central} if its center consists of scalar multiples of the identity element $1$. The  Artin-Whaples Theorem states that the multiplication algebra (i.e., the subalgebra  of {\rm End}$_F(A)$ generated by all multiplication maps $x\mapsto ax$ and $x\mapsto xb$ with $a,b\in A$) of a central simple algebra $A$ acts densely   on $A$. The following corollary to Theorem \ref{T2} can be viewed as an extension of this classical theorem.
  
\begin{corollary}\label{algebraD}  Let
 $A$ be a central simple $F$-algebra such that $1\notin[A,A]$. Then the subalgebra $\mathcal D$  of {\rm End}$_F(A)$ generated by all inner derivations of $A$ acts densely on $[A,A]$. 
\end{corollary} 
 
 Given an arbitrary algebra $A$, it is easy to see that all operators from the algebra $\mathcal D$  generated by inner derivations of $A$ have the form $Tx = \sum_k a_k x b_k$ with $\sum_k a_kb_k =0$ and $\sum_k b_ka_k =0$. 
In \cite{SS} T. Shulman and V. Shulman studied the question whether the converse is true, i.e., does every $T$ of such a form necessarily lie in $\mathcal D$? In particular, they obtained an affirmative answer in the case where $A=M_n(\mathbb C)$  \cite[Theorem 1.11]{SS}. We can now generalize this result as follows.

\begin{corollary}\label{cSS} Assume that {\rm char}$(F) =0$. Let 
 $A$ be a finite dimensional central simple $F$-algebra, and  $\mathcal D$ be the subalgebra of {\rm End}$_F(A)$ generated by all inner derivations of $A$. The following statements are equivalent for 
 $T\in  {\rm End}_F(A)$:
 \begin{enumerate}
 \item[(i)] $T\in\mathcal D$.
  \item[(ii)] There exist $a_k,b_k\in A$ such that $Tx = \sum_k a_k x b_k$, $\sum_k a_kb_k =0$ and $\sum_k b_ka_k =0$. 
 \item[(iii)] $T(1)=0$ and $T(A)\subseteq [A,A]$.
 \end{enumerate}
\end{corollary} 

\begin{proof}
(i)$\Longrightarrow$(ii). Trivial.

(ii)$\Longrightarrow$(iii). From $\sum_k a_kb_k =0$  it follows that $T(1)=0$, and from $\sum_k b_ka_k =0$ it follows that $Tx = \sum_k [a_k x, b_k]\in [A,A]$.

(iii)$\Longrightarrow$(i).  Let us show that there is a direct decomposition $A = [A,A] + F1$.  Denote by $K$ an algebraic closure of $F$, and consider the scalar extension $A\otimes K\cong M_n(K)$.  Since char$(K)=0$, 
for every $a\in A$ the element $a\otimes 1$ can be written as a sum of commutators and a scalar multiple of the identity element: 
$$a\otimes 1 = \sum_i [x_i\otimes \alpha_i, y_i\otimes \beta_i] + 1\otimes \gamma = \sum_i [x_i,y_i]\otimes \alpha_i\beta_i + 1\otimes \gamma.   
$$
Accordingly, $a$ is an $F$-linear combination of elements from $[A,A]$ and $1$. By passing from $A$ to  $A\otimes K$ it is also easy to see that $[A,A]\cap F1 =\{0\}$.

If $T\in  {\rm End}_F(A)$ is such that $T(A)\subseteq [A,A]$, then Corollary \ref{algebraD} tells us that there is $S\in \mathcal D$ such that $S$ and $T$ coincide on $[A,A]$. If we also have $T(1)=0$, then $S=T$. 
\end{proof}

\end{document}